\documentclass[a4paper,11pt]{amsart}
\usepackage{mathptmx}
\usepackage[all]{xy}
\usepackage[colorlinks=true]{hyperref}
\usepackage{enumerate}
\usepackage{amssymb}
\newtheorem{thm}{Theorem}[section]
\newtheorem{thmz}{Theorem}              

\newtheorem{prop}[thm]{Proposition}
\newtheorem{propz}{Proposition}         

\newtheorem{cor}[thm]{Corollary}

\theoremstyle{definition}
\newtheorem{defn}[thm]{Definition}
\newtheorem{ex}[thm]{Example}

\theoremstyle{remark}
\newtheorem{rem}[thm]{Remark}

\renewcommand{\d}{\mathrm d}

\newcommand{\eps}{\varepsilon}

\newcommand{\X}{\ensuremath{\mathfrak{X}}}

\renewcommand{\d}{\mathrm d}               
\newcommand{\Lie}{\boldsymbol{\pounds}}    

\DeclareMathOperator{\Ker}{Ker}         
\DeclareMathOperator{\red}{red}       
\DeclareMathOperator{\graf}{graph}

\newcommand{\al}{\alpha}
\newcommand{\be}{\beta}

\newcommand{\G}{\mathcal{G}}            
\newcommand{\s}{\mathbf{s}}             
\renewcommand{\t}{\mathbf{t}}           
\newcommand{\m}{\mathbf{m}}             
\renewcommand{\gg}{\mathfrak{g}}        
\newcommand{\hh}{\mathfrak{h}}          

\newcommand{\tto}{\rightrightarrows}    

\DeclareMathOperator{\Ad}{Ad}           

\newcommand{\der}{\overrightarrow}

\usepackage[dvips]{graphicx}

\begin{document}

\title{Integrability of Poisson-Lie group actions}
\author{Rui Loja Fernandes}
\address{Departamento de Matem\'{a}tica\\
Instituto Superior T\'{e}cnico\\1049-001 Lisboa\\ Portugal}
\email{rfern@math.ist.utl.pt}
\thanks{R.L.F. is partially supported by the Funda\c{c}\~ao para a Ci\^encia e a Tecnologia
(FCT/Portugal)and project PTDC/MAT/098936/2008. D.I.P. is partially supported by CSIC through
a ``JAE-Doc" research contract, by MICYT (Spain) Grants  MTM2007-62478 and
S-0505/ESP/0158 of the CAM. Both authors are partially supported by CSIC/FCT grant 2007PT0014.}

\author{David Iglesias Ponte}
\address{Instituto de Ciencias Matem\'{a}ticas \\
CSIC-UAM-UC3M-UCM\\
C/ Serrano 123, 28006 Madrid\\ Spain}
\email{iglesias@imaff.cfmac.csic.es}


\begin{abstract}
We establish a 1:1 correspondence between Poisson-Lie group
actions on integrable Poisson manifolds and twisted multiplicative
hamiltonian actions on source 1-connected symplectic groupoids.
For an action of a Poisson-Lie group $G$ on a Poisson manifold
$M$, we find an explicit description of the lifted hamiltonian
action on the symplectic groupoid $\Sigma(M)$. We give
applications of these results to the integration of Poisson
quotients $M/G$, Lu-Weinstein quotients $\mu^{-1}(e)/G$
and Poisson homogeneous spaces $G/H$.
\end{abstract}

\maketitle

\section*{Introduction}             %
\label{sec:introduction}           %

Let $(M,\pi)$ be a Poisson manifold, $G$ a Lie group and assume that $G$ acts on $M$ by Poisson diffeomorphisms.
Such an action usually does not admit a momentum map (e.g., if the action is transverse to the symplectic
leaves). However, there is a symplectization functor which turns this action into a hamiltonian action. More
precisely, associated with an integrable Poisson manifold $(M,\pi)$ there is a canonical symplectic groupoid
$\Sigma(M)\tto M$ and the action lifts to a hamiltonian action on the symplectic groupoid $(\Sigma(M),\Omega)$
with momentum map $J:\Sigma(M)\to\gg^*$ (see \cite{Fer,FerOrRa,Xu0} and references therein).

One should think of $J$ as a canonical momentum map which is
attached to the Poisson action, and which always exists. This map satisfies
\[ J(x\cdot y)=J(x)+J(y),\]
and, in fact, it is the Lie groupoid morphism that integrates the
Lie algebroid morphism $j:T^*M\to\gg^*$ canonically associated
with the action (here we view $\gg^*$ as an abelian Lie algebra).
The momentum map $J$ is relevant, e.g., to understand the reduced
space $M/G$. Namely, $\Sigma(M)//G:=J^{-1}(0)/G$ is a symplectic
groupoid integrating the Poisson quotient $M/G$ (which, however,
does not need to coincide with $\Sigma(M/G)$; see \cite{FerOrRa}).

Our aim is to understand how this theory extends to actions of
Poisson-Lie groups. Suppose $G\times M\to M$ is a Poisson action
of a Poisson-Lie group $G$, with associated Lie bialgebra
$(\gg,\gg^*)$. The map $j:T^*M\to \gg^*$ is still a Lie algebroid
morphism and it integrates to a Poisson groupoid morphism
$J:\Sigma(M)\to G^*$, where $G^*$ is the 1-connected,
dual Poisson-Lie group of $G$.
We would like to lift the $G$-action on $M$ to
a hamiltonian action on $\Sigma(M)$ with momentum map
$J:\Sigma(M)\to G^*$. As was observed by Lu \cite{Lu1}, any
Poisson map $J$ into $G^*$ determines a \emph{local} hamiltonian
$G$-action with momentum map $J$, so there exists a \emph{local}
hamiltonian $G$-action on $\Sigma(M)$.

Recall that a Poisson-Lie group $G$ is said to be {\it complete} if the infinitesimal dressing action of $\gg^*$
on $G$ can be integrated to a global action $G\times G^*\to G$. Our main result is the following:

\begin{thmz}
\label{thm:main}
Let $G$ be a complete Poisson-Lie group,
$M$ an integrable Poisson manifold and $G\times M\to M$ a
Poisson action. There exists a lifted Poisson action of $G$
on the symplectic groupoid $\Sigma(M)$, which is hamiltonian with a
canonical momentum map $J:\Sigma(M)\to G^*$.
\end{thmz}

Our proof of Theorem \ref{thm:main} takes advantage of the
description of $\Sigma(M)$ in terms of cotangent paths (see
\cite{CrFe,CrFe2}) to explicitly construct the lifted $G$-action.
In this respect, it is important to observe that in the situation
described by Theorem \ref{thm:main}, elements of $G$ \emph{do not
act} as groupoid automorphims. In fact, we have the following
result:

\begin{propz}
\label{prop:main:aux:intro} Let $G$ be a connected,
complete Poisson-Lie group, and $\G\tto M$ a source-connected
symplectic groupoid. For a hamiltonian action $G\times\G\to \G$
with momentum map $J:\G\to G^*$ such that $J(M)={e}$, the
following are equivalent:
\begin{enumerate}[(i)]
\item $J:\G\to G^*$ is a groupoid morphism:
\[ J(x\cdot y)=J(x)\cdot J(y),\quad x,y\in\G^{(2)}.\]
\item The twisted multiplicativity property holds:
\begin{equation}
\label{eq:twisted}
g(x\cdot y)=(gx)\cdot (g^{J(x)}y),\quad x,y\in\G^{(2)}, g\in G,
\end{equation}
where we denote by $g^u$ the right dressing action of $u\in G^*$
on an element $g\in G$.
\end{enumerate}
\end{propz}

When the Poisson structure on $G$ vanishes we have $G^*=\gg^*$. In this case, $G^*$ acts trivially on $G$, so
$G$ is always complete. Also, the twisting disappears and the action is by groupoid automorphisms, so we recover
the results mentioned above. The twisted multiplicativity property \eqref{eq:twisted} was also observed by Lu in
\cite{Lu2}.

It is easy to check that, under the assumptions of
Proposition \ref{prop:main:aux:intro}, the twisted multiplicativity
property \eqref{eq:twisted} implies that there is an induced $G$-action
on the units $M$ and that this action is Poisson. Hence,
our results establish a 1:1 correspondence:
\[
\framebox[4.7 cm]{\parbox{4.5 cm}{Poisson actions on integrable Poisson manifolds}}\, \longleftrightarrow \,
\framebox[7.2 cm]{\parbox{7 cm}{Twisted multiplicative hamiltonian actions on source 1-connected symplectic groupoids}}
\]

We illustrate Theorem \ref{thm:main} with a few applications to
the problem of integrating various classes of Poisson manifolds.
The first simple application is:

\begin{thmz}
\label{thm:Poisson:quotients}
Let $G$ be a complete Poisson-Lie group, $(M,\pi)$ an integrable Poisson manifold and $G\times M\to M$ a proper and free Poisson action.
The reduced space
\[\Sigma(M)//G:=J^{-1}(e)/G\]
is a symplectic groupoid integrating the Poisson manifold $M/G$.
\end{thmz}

In general, $\Sigma(M)//G\not=\Sigma(M/G)$. We will see that the failure
in this equality can be controlled in much the same way as in the
case of actions by Poisson diffeomorphisms \cite{FerOrRa} (see
Theorem \ref{thm:int:reduct:alt} below).

The problem of integrating the Poisson quotient $M/G$ has already been discussed by several authors. The first
results in this direction are due to Xu \cite{Xu1}. Later, Lu in \cite{Lu3}, Stefanini in \cite{Ste} and Bonechi
\textit{et al.} \cite{BCST} derive results on integration based on more complicated procedures, using various
notions of action algebroids/groupoids and doubles. Our approach gives a clear explanation for the fact that the
completeness of $G$ implies the existence of a global action on $\Sigma(M)$, rather than just a local one, a
problem circumvented by these authors.

Our second application of Theorem \ref{thm:main} is to the integration of the Poisson manifold obtained by
reduction of a hamiltonian action $G\times M\to M$ with equivariant momentum map $\mu:M\to G^*$. If the action
is proper and free, Lu \cite{Lu1} has shown that the Marsden-Weinstein type quotient $\mu^{-1}(e)/G$ is a
Poisson submanifold of $M/G$. In general, a Poisson submanifold of an integrable Poisson manifold does not need
to be integrable (see \cite{CrFe2}), and when it is integrable, its symplectic groupoid need not be a
subgroupoid of the symplectic groupoid of the ambient Poisson manifold. We will give a simple condition that
guarantees the integrability of $\mu^{-1}(e)/G$ and yields a symplectic groupoid integrating $\mu^{-1}(e)/G$
which is a symplectic subgroupoid of $\Sigma(M)//G$. In the case where $G$ is a 1-connected, simple, compact Lie
group, a classical result of Alekseev \cite{Alek} states that one can gauge transform the Poisson structure so
that the Poisson action becomes an action by Poisson diffeomorphisms. Our methods allow us to describe what
happens to this operation at the level of the lifted action on the symplectic groupoid.

Our last application of Theorem \ref{thm:main} is to the
integrability of Poisson homogeneous spaces. If $G$ is any Poisson-Lie group,
the action of $G$ on itself by left translations is Poisson. Hence, the
theorem shows that it lifts to a hamiltonian $G$-action on $\Sigma(G)$ with
a momentum map $J:\Sigma(G)\to G^*$. If $H\subset G$ is a closed, connected, coisotropic subgroup,
the quotient $G/H$ is a Poisson homogeneous space (\cite{Lu2}). The coisotropy condition is equivalent
to the annihilator $\hh^\perp\subset\gg^*$ of the Lie algebra of $H$ being a Lie subalgebra. We will say
that the pair $(G,H)$ is \emph{relatively complete} if $\hh^\perp\subset\gg^*$ integrates to a closed subgroup
$H^\perp\subset G^*$ and the right dressing action $\gg^*\to\X(G)$ restricted to $\hh^\perp$ integrates to an
action of $H^\perp$ on $H$.

Our methods then lead to a simple proof of the following theorem, which improves results of
\cite{BCST} (see, also, \cite{Lu3} for a more general approach to
the integration of Poisson homogeneous spaces):

\begin{thmz}
Let $G$ be a Poisson-Lie group and let $H\subset G$ be a
closed, coisotropic subgroup, such that $(G,H)$ is relatively complete.
Then $J^{-1}(H^\perp )/H\tto G/H$ is a symplectic groupoid integrating the Poisson homogeneous space $G/H$.
\end{thmz}

The results presented in this paper are part of a wider picture: since Poisson-Lie
groups are the group-like objects in the category of Poisson
groupoids, one should expect them to appear as the \emph{group
of symmetries} of such objects. A systematic study of
symmetries of Poisson groupoids and their infinitesimal
counterparts, Lie bialgebroids, will be the subject of a separate publication
\cite{FP}.

The rest of this paper is organized as follows. In Section \ref{sec:background} we review a few notions and
facts we will need to state and prove our results. Section \ref{sec:actions} contains a proof of Theorem
\ref{thm:main} and other results concerning Poisson actions. Section \ref{sec:applications} contains the three
applications of Theorem \ref{thm:main} mentioned above.

\vskip 10pt \noindent \textbf{Acknowledgments.} The authors would like to thank the following institutions for
their hospitality and support: Erwin Schr\"odinger International Institute for Mathematical Physics (R.L.F. and
D.I.P.), Consejo Superior de Investigaciones Cient\'{\i}ficas (R.L.F.) and Instituto Superior T\'{e}cnico
(D.I.P.). They would also like to thank Yvette Kosmann-Schwarzbach and the anonymous referees for comments and
corrections on a preliminary version of this paper.

\section{Basic Notions}            %
\label{sec:background}             %

In this section we give a short review of all the basic notions we
will need: Poisson and symplectic groupoids, Lie bialgebra(oid)s
and Poisson actions.

\subsection{Poisson and symplectic groupoids}                  %
\label{subsec:groupoids}                                       %

Let $\G$ be a Lie groupoid over $M$. We denote by $\s$ and $\t$
the source and target maps, by $\m:\G^{(2)}\to\G$ the
multiplication (defined on the space $\G^{(2)}$ of pairs of
composable arrows), by $\textbf{i}:\G\to\G$ the inverse map, and
by $\eps:M\to \G$ the identity section. Our convention for the
groupoid multiplication is such that, given two arrows $x,y\in\G$,
the product $x\cdot y:=\m(x,y)$ is defined provided $\s(x)=\t(y)$.
Also, if $m\in M$ we write $1_m:=\eps(m)$ for the unit arrow over
$m$, and if $x\in\G$ we write $x^{-1}:=\textbf{i}(x)$ for the
inverse arrow. We denote the groupoid by $\G\tto M$.

We will be interested in Lie groupoids $\G\tto M$ carrying Poisson structures on
the space of arrows and on the space of units. A \emph{Poisson groupoid} is a
pair $(\G,\Pi)$, where $\G$ is a Lie groupoid and $\Pi\in\X^2(\G)$ is a multiplicative
Poisson structure. This means
that the graph of the groupoid multiplication
\[ \graf(\m):=\{(x,y,x\cdot y) \,|\, \s (x) = \t (y)\}\]
is a coisotropic submanifold of $\G\times \G\times\bar{\G}$ (\cite{We}).
When $\Pi$ is non-degenerate, so $\Omega =\Pi^{-1}$ is a symplectic form,
the multiplicativity condition amounts to:
\begin{equation}
\label{eq:mult:symp}
\m^*\Omega=\pi_1^*\Omega+\pi_2^*\Omega,
\end{equation}
where $\pi_i:\G^{(2)}\to\G$ are the projections on each factor. In
this case, we say that the pair $(\G,\Omega)$ is a
\emph{symplectic groupoid}. A \emph{morphism of Poisson groupoids}
is a Lie groupoid morphism $\Phi:(\G_1,\Pi_1)\to (\G_2,\Pi_2)$
which is also a Poisson map.

For this paper, the two most important examples are the following:

\begin{ex}
Lie groups are precisely the Lie groupoids for which the space of units reduces
to a single object. For a Lie group $G$, a Poisson structure $\pi_G$ is multiplicative
iff the multiplication $m:G\times G\to G$ is a Poisson map (where we consider the
Poisson structure $\pi_G\oplus\pi_G$ on $G\times G$). In this case,
one calls $(G,\pi_G)$ a \emph{Poisson-Lie group}.
\end{ex}

\begin{ex}
Let $(M,\pi)$ be a Poisson manifold. Its \emph{Weinstein groupoid} $\Sigma(M)\tto M$
is defined as (see \cite{CrFe2}):
\[ \Sigma(M):=\frac{\{\text{cotangent paths}\}}{\{\text{cotangent homotopies}\}},\]
where multiplication is concatenation of cotangent paths. If $p:T^*M\to M$ denotes the projection,
the source and target maps are given by:
\[ \s([a])=p(a(0)),\quad \t([a])=p(a(1)).\]

A Poisson manifold $(M,\pi)$ is called integrable if its Weinstein groupoid $\Sigma(M)$ is
smooth (in which case, one has $\dim \Sigma(M)=2\dim M$). The obstructions to integrability
were determined in \cite{CrFe,CrFe2}. When $(M,\pi)$ is integrable, $\Sigma(M)$ carries a
natural multiplicative symplectic form $\Omega$. Moreover, the source (respectively, target map)
is a Poisson (resp.~anti-Poisson map).
\end{ex}

Alan Weinstein in \cite{We} observed that the properties of the source/target maps in the last
example are by no means exceptional: given a Poisson groupoid $(\G ,\Pi)$ with base $M$ there exists
a unique Poisson structure $\pi$ on $M$, such that $\s :\G \to M$ is a Poisson map and $\t :\G \to M$ is anti-Poisson.

\subsection{Lie bialgebras and Lie bialgebroids}               %
\label{subsec:bialgebroids}                                    %

Now let us go to the infinitesimal level. We will denote by $A$ a Lie algebroid with bundle projection $p:A\to
M$, anchor $\#:A\to TM$, and Lie bracket $[~,~]_A$ on its space of sections. The $A$-differential forms are
$\Omega^\bullet(A):=\Gamma(\wedge^\bullet A^*)$ and they form a complex with the $A$-differential $\d_A$ (see,
e.g., \cite{Mc2}). Our conventions are such that if $\G\tto M$ is a Lie groupoid, then its Lie algebroid
$A=A(\G)$ has $A_x:=\Ker\d_{1_x}\s$ and $\#_x:=\d_{1_x}\t$. Moreover, $\Gamma(A)=\X(A)$ is identified with the
space $\X_r(\G)$ of right invariant vector fields on $\G$ and we will denote by $\der{X}\in \X_r(\G)$ the right
invariant vector field corresponding to $X\in \Gamma(A)$. Similarly, one obtains identifications of
$\X^\bullet(A)$ and $\Omega^\bullet(A)$ with the spaces $\X^\bullet_r(\G)$ and $\Omega^\bullet_r(\G)$ of right
invariant multivector fields and differential forms on $\G$ (note that a right invariant differential form is a
$\s$-foliated differential form on $\G$). Under these identifications, the bracket $[~,~]_A$ and the
differential $\d_A$ are identified with the Schouten bracket on right invariant multivector fields and the de
Rham differential on right invariant differential forms.

We recall the following basic proposition, due to Weinstein \cite{We}:

\begin{prop}\label{prop:induced-bialgebroid}
If $(\G,\Pi)$ is a Poisson groupoid then it induces a Lie algebroid structure
on $A(\G)^*$, the dual of the Lie algebroid, whose $A(\G)^*$-differential
is characterized by
\begin{equation}\label{eq:dual-differential}
 \der{\d_{A(\G)^*}X} = -[\der{X},\Pi] ,\qquad (X\in \X (A)).
\end{equation}
\end{prop}

This leads to the infinitesimal counterpart of a Poisson groupoid,
i.e., the notion of a \emph{Lie bialgebroid}. This is pair of Lie
algebroid structures $(A,A^*)$, on a vector bundle $A\to M$ and on
its dual bundle $A^*\to M$, such that for any $X,\ Y\in \X(A)$,
\[
\d_{A^*}[X, Y]_A = \Lie _{X}\d_{A^*}Y-\Lie_{Y}\d_{A^*}X.
\]
Just like for a Poisson groupoid, if $(A, A^*)$ is a Lie bialgebroid over $M$,
there exists a Poisson structure $\pi _M$ on $M$ which is characterized by
\[
\pi _M (df ,dg )= \# (\d_{A^*}f) (g)=\langle \d_{A^*}f, \d_A g\rangle , \qquad (f,g \in C^\infty (M)).
\]
A \emph{morphism} of Lie bialgebroids $\phi:(A_1,A_1^*)\to (A_2,A_2^*)$ is a Lie algebroid morphism
$\phi:A_1\to A_2$ which is also a Poisson map (note that $A_i$ has a fiberwise linear Poisson structure,
being the dual of the Lie algebroid $A_i^*$).

If $(\G,\Pi)$ is a Poisson groupoid, it follows from Proposition \ref{prop:induced-bialgebroid} that
$(A(\G),A(\G)^*)$ is a Lie bialgebroid. Conversely, it is proved in
\cite{McXu2} that any Lie bialgebroid structure $(A,A^*)$, where
$A$ can be integrated to a Lie groupoid, can actually be integrated to a Poisson
groupoid $(\G(A),\Pi )$. Here $\G (A)$ is just the source 1-connected Lie
groupoid integrating $A$. In this situation, the Poisson structures on $M$ induced by
$(\G(A),\Pi)$ and $(A,A^*)$ coincide. Similarly, Poisson groupoid morphisms $\Phi:\G_1\to \G_2$
are in 1:1 correspondence with Lie bialgebroid morphisms $\phi:(A_1,A_1^*)\to (A_2,A_2^*)$, provided
$\G_1$ is source 1-connected.

Note that the notion of Lie bialgebroid is symmetric: if $(A,A^*)$
is a Lie bialgebroid so is $(A^*,A)$. On the other hand, at the
level of groupoids things are more subtle: for example, in a Lie
bialgebroid $(A,A^*)$ it is possible that $A$ is integrable while
$A^*$ is not.

\begin{ex}
It is well known that if $(M,\pi)$ is a Poisson manifold, then $A=T^*M$ is Lie algebroid with
anchor $\pi^\sharp:T^*M\to TM$ and Lie bracket:
\[ [\al,\be]=\Lie_{\pi^\sharp(\al)}\be-\Lie_{\pi^\sharp(\be)}\al-\d(\pi(\al,\be)).\]
When one consider $A^*=TM$ with its canonical Lie algebroid structure, the pair $(T^*M,TM)$ becomes
a Lie bialgebroid. While $A^*=TM$ is always integrable, $A=T^*M$ does not have to be integrable. Its
integrability is equivalent to the integrability of $(M,\pi)$. When $(M,\pi)$ is integrable, $(\Sigma(M),\Omega)$
is the source 1-connected symplectic groupoid integrating the Lie bialgebroid $(T^*M,TM)$.
\end{ex}

\begin{ex}
If $(G,\pi_G)$ is a Poisson-Lie group, the corresponding Lie
bialgebroid is just a Lie bialgebra $(\gg,\gg^*)$. According to
our conventions, $\gg$ is the space of right invariant vector
fields on $G$. We can also identify $\gg^*$ with the space of
right invariant 1-forms on $G$. The bracket on 1-forms induced by
$\pi_G$ (see the previous example) preserves the right invariant
forms, and it induces the Lie bracket $[~,~]_{\gg^*}$ on $\gg^*$.

The 1-connected Lie group integrating $\gg^*$, denoted $G^*$, is called the dual Poisson-Lie group: its Lie
bialgebra is $(\gg^*,\gg)$.
\end{ex}
\subsection{Poisson actions}                                   %
\label{subsec:poisson:actions}                                 %

Let $(G,\pi_G)$ be a Poisson-Lie group and let $(M,\pi)$ be a Poisson manifold. Recall that a smooth action
$\Psi:G\times M\to M$ is called a \emph{Poisson action} if $\Psi$ is a Poisson map. Here the product $G\times M$
is equipped with the direct sum Poisson structure $\pi_G\oplus\pi$.

For a smooth action $\Psi:G\times M\to M$ of a Lie group on a manifold $M$, we will denote by
$\psi:\gg\to\X(M)$ the corresponding infinitesimal Lie algebra action defined by:
\[ \psi(\xi)_a=\left.\frac{\d}{\d t}\exp(t\xi)a\right|_{t=0}\quad (\xi\in\gg).\]
According to our conventions, $\gg$ is identified with the space of right invariant vector fields on $G$ and it
follows that $\psi:\gg\to \X(M)$ is a Lie algebra homomorphism. The following characterization of Poisson
actions is due to Lu \cite{Lu,Lu1}:

\begin{prop}
\label{prop:Poisson:actions:1}
Let $(G,\pi_G)$ be a connected Poisson-Lie group and let $(M,\pi)$ be a Poisson manifold.
For a smooth action $\Psi:G\times M\to M$ the following two properties are equivalent:
\begin{enumerate}[(i)]
\item The action $\Psi$ is Poisson;
\item Setting $\delta:=\d_e\pi:\gg\to\gg\wedge\gg$, the infinitesimal action satisfies:
\[ \Lie_{\psi(\xi)}\pi=(\psi\wedge\psi)\delta(\xi),\quad (\xi\in\gg).\]
\end{enumerate}
\end{prop}

The map $\delta$ is just (the dual of) the Lie bracket on $\gg^*$. Hence, the proposition leads to a definition
of an infinitesimal action of a Lie bialgebra $(\gg,\gg^*)$ on a Poisson manifold $(M,\pi)$.

\begin{ex}
Let $(G,\pi_G)$ be a Poisson-Lie group with Lie bialgebra $(\gg,\gg^*)$. According to our conventions, we can
identify $\gg^*$ with the space of right invariant 1-forms on $G$. The map $\lambda:\gg^*\to\X(G)$ which to a
right invariant 1-form $\eta\in\gg^*$ associates the vector field $\pi_G^\sharp(\eta)$ is a Lie algebra morphism
and so defines an (left) infinitesimal action of $\gg^*$ on $G$. Using Proposition \ref{prop:Poisson:actions:1},
one checks that this is an infinitesimal action of $(\gg^*,\gg)$ on the Poisson manifold $(G,\pi_G)$, called the
\emph{left dressing action}. Similarly, the identification of $\gg^*$ with the left invariant 1-forms on $G$,
leads to Lie algebra anti-morphism $\rho:\gg^*\to\X(G)$ and hence to a \emph{right dressing action}. Switching
the roles of $G$ and $G^*$ we also obtain left/right dressing actions of $\gg$ on $G^*$. If one of the
infinitesimal dressing actions is complete so is the other. We say that $(G,\pi_G)$ is a \emph{complete
Poisson-Lie group} if the right dressing action $\rho:\gg ^*\to\X(G)$ integrates to a (Poisson) right action of
$(G^*,\pi_{G^*})$ on $(G,\pi_G)$.
\end{ex}

There is another useful characterization of Poisson actions, due to Xu \cite{Xu0}:

\begin{prop}
\label{prop:Poisson:actions:2}
Let $(G,\pi_G)$ be a connected Poisson-Lie group and let $(M,\pi)$ be a Poisson manifold. For a smooth action
$\Psi:G\times M\to M$ define $j:T^*M\to\gg^*$ by
\[ \langle j(\al),\xi\rangle=\langle \al,\psi(\xi)\rangle,\quad (\xi\in\gg).\]
Then the following two properties are equivalent:
\begin{enumerate}[(i)]
\item The action $\Psi$ is Poisson;
\item The map $j:T^*M\to\gg^*$ is a Lie bialgebroid morphism.
\end{enumerate}
\end{prop}

\subsection{Hamiltonian actions}                               %
\label{subsec:hamiltonian}                                     %

Let $\Psi:G\times M\to M$ be a Poisson action. A smooth map
$\mu:M\to G^*$ is called a \emph{momentum map} for the
action if:
\begin{equation}
\label{eq:momentum}
\psi(\xi)=\pi^{\sharp}(\mu^*\xi^R)\quad (\xi\in\gg).
\end{equation}
Here, $\xi^R\in\Omega^1(G^*)$ is the right invariant 1-form on $G^*$ with value $\xi\in\gg$ at the identity
$e\in G^*$. Observe that when $\pi_G=0$, so that $G^*=\gg^*$, the momentum map condition \eqref{eq:momentum}
reduces to the usual condition. These generalized momentum maps were first studied by Lu in \cite{Lu,Lu1}.

\begin{ex}
From the definition, it is evident that the left/right dressing actions of $G$ on $G^*$ have momentum map
$G^*\to G^*$ the identity map. Similarly, the left/right dressing actions of $G^*$ on $G$ have momentum map
$G\to G$ the identity. These actions satisfy versions of the twisted multiplicativity property
\eqref{eq:twisted}. For example, the (left) dressing action $G\times G^*\to G^*$ satisfies:
\begin{equation}
\label{eq:twist:dress}
g(u_1\cdot  u_2)=(g u_1)\cdot(g^{u_1} u_2),\qquad (g\in G,\ u_1,u_2\in G^*).
\end{equation}
\end{ex}

A proof of the following basic fact can be found in \cite{Lu1}:

\begin{prop}
Let $(G,\pi_G)$ be a connected and complete Poisson-Lie group. A
momentum map $\mu:M\to G^*$ for a Poisson action $G\times
M\to M$ is $G$-equivariant (relative to the left dressing action
of $G$ on $G^*$) if and only if it is a Poisson map.
\end{prop}

We will say that a Poisson action $G\times M\to M$ is a \emph{hamiltonian action} if it admits an equivariant
momentum map $\mu:M\to G^*$. Lu has also shown that the usual Marsden-Weinstein symplectic reduction extends to
these hamiltonian actions.

In order to explain this fact, let $G\times M\to M$ be a hamiltonian action on a Poisson manifold, with momentum
map $\mu:M\to G^*$. If $u\in G^*$, denote by $G_u$ the isotropy group of $u$ for the the left dressing action of
$G$ on $G^*$. Then we have the following result (see Lu \cite{Lu1}):

\begin{thm}
Let $G\times M\to M$ be a proper and free hamiltonian action, with momentum map $\mu:M\to G^*$.
For each $u\in G^*$, the level set $\mu^{-1}(u)$ carries a
natural Dirac structure $L_u$, the space
$\mu^{-1}(u)/G_u$ carries a natural Poisson structure and
we have a commutative diagram:
\[
\xymatrix{
     &M\ar[dr]\\
\mu^{-1}(u)\ar[ur]\ar[dr]& &M/G\\
 &\mu^{-1}(u)/G_u\ar[ur]}
\]
where the inclusions are backward Dirac maps and the projections are forward Dirac maps (see \cite{BuRa} for the
definition of these classes of maps).
\end{thm}

\begin{rem}
If one starts with a hamiltonian action on a \emph{symplectic} manifold $(S,\omega)$ the reduced spaces
$\mu^{-1}(u)/G_u$ are also symplectic. In fact, their connected components are the symplectic leaves of the
quotient Poisson manifold $S/G$.
\end{rem}

If we start with a Poisson-Lie group $(G,\pi_G)$ and a Poisson manifold $(M,\pi)$, any Poisson map $\mu:M\to
G^*$ determines an infinitesimal Poisson action $\psi:\gg\to\X(M)$ by setting:
\[ \psi(\xi):=\pi^{\sharp}(\mu^*\xi^R)\quad (\xi\in\gg). \]
Integration gives a \emph{local} Poisson action with equivariant
momentum map $\mu$.

\section{Integration of Poisson actions}%
\label{sec:actions}                %

In this section, we will prove Theorem \ref{thm:main} and other
results concerning the integration of Poisson actions.

\subsection{Poisson actions on symplectic groupoids}           %
\label{subsec:twisted:action}                                  %

Before considering the problem of lifting a Poisson action on $M$
to a Poisson action on the symplectic groupoid $\Sigma(M)$, we
discuss actions on symplectic groupoids and how the twisted
multiplicativity property arises.

\begin{prop}
\label{prop:main:aux}
Let $(G,\pi_G)$ be a connected, complete
Poisson-Lie group, and let $\G\tto M$ be a source-connected symplectic
groupoid. If $G\times\G\to \G$ is a hamiltonian action with momentum
map $J:\G\to G^*$ such that $J(M)={e}$, the following are
equivalent:
\begin{enumerate}[(i)]
\item $J:\G\to G^*$ is a groupoid morphism:
\[ J(x\cdot y)=J(x)\cdot J(y),\quad x,y\in\G^{(2)}.\]
\item The twisted multiplicativity property holds:
\begin{equation}
\label{eq:twisted:2}
g(x\cdot y)=(gx)\cdot (g^{J(x)}y),\quad x,y\in\G^{(2)}, g\in G,
\end{equation}
where we denote by $g^u$ the right dressing action of $u\in G^*$
on an element $g\in G$.
\end{enumerate}
\end{prop}

\begin{proof}
Denote by $\psi:\gg\to\X(\G)$ the infinitesimal $\gg$-action.
For the proof, we remark that the multiplicativity property \eqref{eq:mult:symp} of the symplectic form $\Omega$, when
evaluated at $(x,y)$ on the pair $(\psi(\xi)_x,\psi(\Ad^*J(x)\cdot \xi)_y),(v,0)\in T_{(x,y)}\G^{(2)}$ yields:
\begin{equation}
\label{eq:mult:aux}
\Omega_{x\cdot y}(\d_{(x,y)}\m(\psi(\xi)_x,\psi(\Ad^* J(x)\cdot \xi)_y),\d_x R_{y}v)=\Omega_x(\psi(\xi)_x,v).
\end{equation}
where $\m:\G^{(2)}\to\G$ is the groupoid multiplication and $R_y$ denotes right translation by the element $y\in\G$ (here $v$
is any vector tangent to the source fiber at $x$).

Now, since the Lie group $G$ is connected, the twisted
multiplicativity property \eqref{eq:twisted:2} is equivalent to
its infinitesimal version, which reads:
\begin{equation}
\label{eq:twisted:inft}
\psi(\xi)_{x\cdot y}=\d_{(x,y)}\m(\psi(\xi)_x,\psi(\Ad^* J(x)\cdot \xi)_y),\quad x,y\in\G^{(2)}, \xi\in\gg.
\end{equation}
So if this condition holds, we conclude from \eqref{eq:mult:aux} that
\[ \Omega_{x\cdot y}(\psi(\xi)_{x\cdot y},\d_x R_{y}v)=\Omega_x(\psi(\xi)_x,v),\]
for any vector $v$ tangent to the source fiber at $x$. In other
words, $i_{\psi(\xi)}\Omega=J^*\xi^R$ is a right invariant 1-form,
for all $\xi\in\gg$. But if $J^*\xi^R$ is a right invariant
1-form, for all $\xi\in\gg$ and $J(M)=e$, then $J:\G\to G^*$ must
be a groupoid homomorphism.

Conversely, assume that $J:\G\to G^*$ is a groupoid homomorphism.
Then:
\begin{equation}\label{eq:momentum:Poisson}
\psi(\xi)=\pi^{\sharp} (J^*\xi^R),
\end{equation}
where $\pi=\Omega^{-1}$ is a multiplicative Poisson structure. Let $\xi \in \gg$ and $(x,y)\in\G^{(2)}$. Since
$J$ is a groupoid morphism, we have
\begin{equation}\label{eq:cotangent:morphism}
J^*(\theta _1\star_{G^*} \theta _2)=(J^*\theta _1)\star_{\G}
(J^*\theta _2), \qquad (\theta _1\star \theta _2)\in
(T^*(G^*))^{(2)},
\end{equation}
where $\star _{\G}$ (respectively,  $\star _{G^*}$) denotes the groupoid multiplication in $T^*\G$
(respectively, $T^*(G^*)$). Now, using \eqref{eq:momentum:Poisson}, \eqref{eq:cotangent:morphism} and the fact
that $(\xi ^R)_{uv}=(\xi^R)_u\star _{G^*}(\Ad^*J(x)\cdot\xi^R)_v$ for $\xi\in \gg$ and $u,v\in G^*$, we deduce
\begin{align*}
\psi (\xi )_{x\cdot y}&=\pi ^\sharp ( J^* \xi ^R)_{x\cdot y}
=\pi ^\sharp ( J^*(( \xi ^R)_{J(x)} \star_{G^*}  ((\Ad^*J(x)\cdot \xi )^R)_{J(y)}))\\
&=\pi ^\sharp ( (J^* \xi ^R)_x \star_{\G}  (J^* (\Ad^*J(x)\cdot \xi )^R)_y )\\
&=\d_{(x,y)}\m(\pi^\sharp (J^* \xi ^R)_x,\pi ^\sharp (J^*(\Ad^*J(x)\cdot \xi)^R)_y)\\
&=\d_{(x,y)}\m(\psi (\xi )_x,\psi (\Ad^*J(x)\cdot \xi )_y) .
\end{align*}
Here, we have also used that $\pi ^\sharp :T^*\G\to T\G$ is a
groupoid morphism, i.e.,
\[
\d\m (\pi^\sharp (\eta _1),\pi ^\sharp (\eta _2))=\pi ^\sharp
(\eta _1\star_{\G} \eta _2),\qquad (\eta _1,\eta _2)\in
(T^*\G)^{(2)}.
\]
Therefore, the infinitesimal condition \eqref{eq:twisted:inft} is
satisfied and, as a consequence, the twisted multiplicativity
condition holds.
\end{proof}

Our next remark is even more general.

\begin{prop}
\label{prop:general}
Let $(G,\pi_G)$ be a complete Poisson-Lie
group, $G\times \G\to\G$ a smooth action on a Lie groupoid, and
$J:\G\to G^*$ a groupoid morphism. If the action satisfies the
twisted multiplicativity property \eqref{eq:twisted:2}, then there
is an induced action on the Lie algebroid $A$ of $\G$. Moreover,
if $\G$ is source 1-connected the $G$-action on $\G$ is completely
determined by $J$ and the induced $G$-action on $A$.
\end{prop}

\begin{rem}
Note that in this proposition there is no assumption about a
symplectic or Poisson structure on $\G$. Also, the induced action
on $A$, in general, \emph{is not} by Lie algebroid automorphisms.
\end{rem}

\begin{proof}
First, we remark that the twisted multiplicativity property \eqref{eq:twisted:2} and the fact that $J$ is a
homomorphism imply that, for any $x\in\G$, we have:
\[
\left\{
\begin{array}{l}
g\, x=g\,(1_{\t(x)}\cdot x)=(g\,1_{\t(x)})\cdot(g\,x)\\
g\, x=g\,(x\cdot 1_{\s(x)})=(g\,x)\cdot(g^{J(x)}\,1_{\s(x)})
\end{array}\right.
\quad\Rightarrow\quad
\left\{
\begin{array}{l}
g\,1_{\t(x)}=1_{\t(g\,x)}\\
g^{J(x)}\,1_{\s(x)}=\,1_{\s(g\,x)}
\end{array}\right.\]
Therefore, we have an induced $G$-action on $M$ such that:
\[ g\,1_m=1_{g\,m}.\]
Moreover, we also find that:
\[ \t(g\,x)=g\,\t(x),\quad \s(g\,x)=g^{J(x)}\s(x).\]
We will also need the identity:
\[ (g\, x)^{-1}=g^{J(x)}\,x^{-1},\]
whose proof is straightforward from \eqref{eq:twisted:2} and the fact that $J$ is a homomorphism.

The previous identities show that the $G$-action sends $\t$-fibers
to $\t$-fibers, but does not preserve source fibers. However, we
can consider a new $G$-action on $\G$ defined by:
\[ g\odot x:=(g\,x^{-1})^{-1},\]
which does preserve $\s$-fibers, and induces the same action on
the identity section. Hence, we have an induced $G$-action on the
Lie algebroid $A$ of $\G$, by vector bundle automorphisms (but, in
general, not Lie algebroid automorphisms), defined by:
\[ g\,a:=\left.\frac{\d}{\d t}g\odot \gamma(t)\right|_{t=0},\quad (g\in G, a\in A_m)\]
where $\gamma(t)$ is any curve lying in the source fiber
$\s^{-1}(m)$ with $\gamma(0)=1_m$ and $\dot{\gamma}(0)=a$.

If $\G$ has source 1-connected fibers, then we can identify an element
$x\in\G$ with the homotopy class $[x(t)]$, where $x(t)$ is any
$\s$-path, i.e., a path lying in the source fiber through $x$ and
such that $x(0)=1_{\s(x)}$ and $x(1)=x$ (see \cite{CrFe}). Then we can identify $\G$
with the Weinstein groupoid $\G(A)$ consisting of $A$-paths modulo
$A$-homotopy. This identification can be done at the level of
paths by setting:
\[ x(t) \longmapsto a(t):=\left.\frac{\d}{\d s}x(s)\cdot x(t)^{-1}\right|_{s=t}.\]
Using this identification, we transport the $G$-action on
$\G$ to an action on $\G(A)$: if $x$ is represented by the
$\s$-path $x(t)$ then $g\,x$ is represented by the $\s$-path:
\[ \bar{x}(t):=g^{J(x)J(x(t))^{-1}}\,x(t). \]
In fact, we find $\s(\bar{x}(t))=g^{J(x)}\,\s(x(t))=g^{J(x)}\s(x)=\s(g\,x)$ and $\bar{x}(1)=g\,x$. Then we
compute the $A$-path associated to $\bar{x}(t)$,
\begin{align*}
\bar{a}(t):=&\left.\frac{\d}{\d s}\bar{x}(s)\cdot \bar{x}(t)^{-1}\right|_{s=t}\\
=&\left.\frac{\d}{\d s} \left(g^{J(x)J(x(s))^{-1}}\,x(s)\right)\cdot \left(g^{J(x)J(x(t))^{-1}}\,x(t)\right)^{-1}\right|_{s=t}\\
=&\left.\frac{\d}{\d s} \left(g^{J(x)J(x(t))^{-1}}\,\left(x(s)\cdot x(t)^{-1}\right)^{-1}\right)^{-1}\right|_{s=t}\\
=&\left.\frac{\d}{\d s} \left(g^{J(x)J(x(t))^{-1}}\,
\left(x(s)\cdot
x(t)^{-1}\right)\right)\right|_{s=t}=g^{J(x)J(x(t))^{-1}}a(t).
\end{align*}
This last expression shows that the action of $G$ on $\G$ is
completely determined by $J$ and the action of $G$ on $A$, as claimed.
\end{proof}

\subsection{Lifting of local Poisson actions}                  %
\label{subsec:lifted:action:local}                             %

Let us now consider the problem of lifting a Poisson action on $M$
to a Poisson action on the symplectic groupoid $\Sigma(M)$.

Given a Poisson action $\Psi:G\times M\to M$, it follows from
Proposition \ref{prop:Poisson:actions:2} that the induced map
$j:T^*M\to\gg^*$ is a Lie bialgebroid morphism from $(T^*M,TM)$ to
$(\gg^*,\gg)$. Integrating this morphism (see \cite[Theorem 5.5]{Xu0}),
we conclude that:

\begin{cor}
\label{cor:integ:J} Let $\Psi:G\times M\to M$ be a Poisson action
of a Poisson-Lie group $(G,\pi_G)$ on an integrable Poisson
manifold $(M,\pi)$. The Lie bialgebroid morphism $j:T^*M\to\gg^*$
integrates to a morphism of Poisson groupoids $J:\Sigma(M)\to
G^*$.
\end{cor}

At the level of cotangent paths, the map $J$ is simply given by the formula:
\[ J([a])=[j\circ a]\]
(see \cite{CrFe}, where it is explained how to integrate morphisms
of Lie algebroids to morphisms of Lie groupoids in terms of
cotangent paths).

Since $J:\Sigma(M)\to G^*$ is a Poisson map and a groupoid morphism, we conclude from Proposition \ref{prop:main:aux} that:

\begin{prop}
\label{prop:local:act}
Let $\Psi:G\times M\to M$ be a Poisson action of a Poisson-Lie
group $(G,\pi_G)$ on a Poisson manifold $(M,\pi)$. There exists a
local hamiltonian action of $G$ on $\Sigma(M)$ with momentum map
$J:\Sigma(M)\to G^*$ which satisfies the infinitesimal twisted
multiplicativity property \eqref{eq:twisted:inft}.
\end{prop}

Later, we will give an explicit expression for this local action (see Remark \ref{rem:local:action}).
The following example shows that, in general, the lifted action will not be a \emph{global action}.

\begin{ex}
Let $G$ be any Poisson-Lie group which is not complete. The action of $G$ on itself by left translations
$G\times G\to G$ is a Poisson action. The lifted (local) action on $\Sigma(G)$ is not a global action. In fact,
observe that the identity $e\in G$ is a fixed point for the Poisson structure where the isotropy Lie algebra is
$\gg^*$. Hence, the corresponding isotropy group is:
\[ \Sigma(G)_e=\s^{-1}(e)=\t^{-1}(e)\simeq G^*.\]
The restriction of $J:\Sigma(G)\to G^*$ to this isotropy group is an isomorphism, so if the lifted action were a
global action, the dressing action would have to be complete.
\end{ex}

\subsection{Lifting to global Poisson actions}                 %
\label{subsec:lifted:action}                                   %

Our main result states that if $(G,\pi_{G})$ is a complete Poisson-Lie group, then the lifted action is a global
action. In the sequel, we will assume that $G$ is complete and will denote by $g^u$ the right dressing action of
an element $u\in G^*$ on an element $g\in G$.

\begin{thm}
\label{thm:lift:actions} Let $(G,\pi_G)$ be a complete Poisson-Lie
group, $(M,\pi)$ an integrable Poisson manifold and $G\times M\to
M$ a Poisson action. Let $J:\Sigma(M)\to G^*$ be the integration
of $j:T^*M\to\gg^*$. Then there exists a lifted hamiltonian action
of $(G,\pi_G)$ on the symplectic groupoid $\Sigma(M)$ with
momentum map $J$, such that:
\begin{enumerate}
\item $J$ is equivariant:
\[ J(g\,x)=g\,J(x),\qquad (g\in G, x\in\Sigma(M)). \]
\item The action is twisted multiplicative:
\[
g\,(x\cdot y)=(g\, x)\cdot (g^{J(x)}\, y),\qquad (g\in G,(x,y)\in\Sigma(M)^{(2)}).
\]
\end{enumerate}
\end{thm}

\begin{proof}
Let $a:I\to T^*M$ be a cotangent path and define a new path $\bar{a}:I\to T^*M$ by:
\begin{equation}
\label{cotg:path}
\overline{a}(t):=g^{J(x)J(x(t))^{-1}}\, a(t).
\end{equation}
In this formula, we use the lifted cotangent action of $G$ on
$T^*M$ and $x(t)$ denotes the element in $\Sigma(M)$ which, for a
fixed $t\in I$, is defined by the cotangent path $s\mapsto t\,a(st)$. The motivation
for this definition can be found in the proof of Proposition \ref{prop:general}.

One now checks that:
\begin{enumerate}[(a)]
\item For any cotangent path $a(t)$, the path $\bar{a}(t)$ defined by \eqref{cotg:path}
is also a cotangent path.
\item If $a_{\eps}$ is a cotangent homotopy, then the corresponding family $\bar{a}_\eps$
defined by \eqref{cotg:path} is also a cotangent homotopy.
\end{enumerate}

This means that formula \eqref{cotg:path} leads to a map $G\times \Sigma(M)\to\Sigma(M)$ by setting
at the level of cotangent homotopy classes:
\begin{equation}
\label{eq:action} g\, [a(t)]=[g^{J(x)J(x(t))^{-1}}\, a(t)].
\end{equation}

Using this formula, we will show that $J:\Sigma(M)\to G^*$ is $G$-equivariant, i.e., that:
\begin{equation}
\label{eq:equivariance} J(g\, x)=g\, J(x), \quad (x\in\Sigma(M),\
g\in G),
\end{equation}
where on the left-hand side $g$ acts by \eqref{eq:action} and on the
right-hand side $g$ acts by the (left) dressing action on $G^*$.

In order to prove \eqref{eq:equivariance} one starts by remarking that, since $G^*$ is simply-connected, one can
identify $G^*$ with paths $\xi:I\to \gg^*$ up to $\gg^*$-homotopy. This is a special instance of the general
construction mentioned in the proof of Proposition \ref{prop:general}: given an element $u\in G^*$ we first
identify it with (the homotopy class of) a path $u(t)\in G^*$ starting at the identity $u(0)=e$ and ending at
$u(1)=u$. Then we associate to it (the $\gg^*$-homotopy class of) a path in the Lie algebra $\xi:I\to \gg^*$ by
setting:
\[ \xi(t)=\left.\frac{\d}{\d s}u(s)u(t)^{-1}\right|_{s=t}.\]
By the formula proved at the end of Proposition \ref{prop:general}, under this identification, the dressing
action $G\times G^*\to G^*$ is given at the level of $\gg^*$-paths by:
\[ g\, [\xi(t)]=[\Ad^*g^{u(1)u(t)^{-1}}\, \xi(t)].\]
Using this relation, we see that if $x=[a(t)]\in\Sigma(M)$ and $g\in G$, then:
\begin{align*}
J(g\, x)&=[j(g^{J(x)J(x(t))^{-1}}\, a(t))]\\
&=[\Ad^*g^{J(x)J(x(t))^{-1}}\, j(a(t))]\\
&=g\, [j(a(t))]=g\, J(x),
\end{align*}
so the equivariance follows.

We now check that \eqref{eq:action} defines a $G$-action, i.e., that:
\begin{enumerate}[(a)]
\item If $e\in G$ is the identity element, then $e\, x=x$, for all
$x\in\Sigma(M)$; \item If $g,h\in G$, then $g\,(h\, x)=(gh)\, x$,
for all $x\in\Sigma(M)$;
\end{enumerate}

In order to prove that (a) holds, one observes that the identity element $e\in G$ is fixed by the right dressing
action of $G^*$, so if $x=[a(t)]\in\Sigma(M)$ we find:
\[ e\, [a(t)]=[e^{J(x)J(x(t))^{-1}}\, a(t)]=[a(t)],\]
and (a) follows.

To prove (b), we first observe that if $x=[a(t)]\in\Sigma(M)$ so
that $h\, x=[h^{J(x)J(x(t))^{-1}}\, a(t)]$, then:
\[ J(h\, x)=h\, J(x),\quad J((h\, x)(t))=h^{J(x)J(x(t))^{-1}}J(x(t)).\]
It then follows from the twisted multiplicativity of the left dressing action of $G$ on $G^*$ that:
\begin{align*}
g\, (h\, x)&=g\,(h\, [a(t)])\\
&=[g^{J(h\, x)J((h\, x)(t))^{-1}}h^{J(x)J(x(t))^{-1}}\, a(t)]\\
&=[g^{(h\, J(x))(h^{J(x)}\, J(x(t))^{-1})}h^{J(x)J(x(t))^{-1}}\, a(t)]\\
&=[g^{h\, (J(x) J(x(t))^{-1})}h^{J(x)J(x(t))^{-1}}\, a(t)].
\end{align*}
Using the twisted multiplicativity of the right dressing action of $G^*$ on $G$ we obtain:
\begin{align*}
g\,(h\, x)&=g\,(h\, [a(t)])\\
&=[g^{h\, (J(x) J(x(t))^{-1})}h^{J(x)J(x(t))^{-1}}\, a(t)]\\
&=[(gh)^{J(x)J(x(t))^{-1}}\, a(t)]\\
&=(gh)\,[a(t)]=(gh)\, x,
\end{align*}
so (b) holds.

Finally, we need to check that the $G$-action \eqref{eq:action} is hamiltonian with momentum map $J$. By
Proposition \ref{prop:local:act} we have a local hamiltonian $G$-action on $\Sigma(M)$ with momentum map $J$.
This local action satisfies the twisted multiplicativity property and the induced action on the Lie algebroid
$A(\Sigma(M))=T^*M$ is the cotangent lifted action. On the other hand, the $G$-action \eqref{eq:action} also
satisfies the twisted multiplicativity property and induces the same action on $T^*M$. By the uniqueness
property proved in Proposition \ref{prop:general}, the two actions must coincide, so we conclude that the lifted
action is hamiltonian with momentum map $J$.
\end{proof}

\begin{rem}
\label{rem:local:action} Assume that $G$ is not a complete Poisson-Lie group. The identity $e\in G$ is a fixed
point for the infinitesimal right dressing action of $G^*$ on $G$. It follows that, for any element $u\in G^*$,
the (local) dressing action $g^u$ is defined for $g$ in a sufficiently small neighborhood of $e\in G$ (which
depends on $u$). Similarly, if $u:I\to G^*$ is any path starting at $u(0)=e\in G^*$, a compactness argument
shows that $g^{u(1) u(t)^{-1}}$ is defined provided $g$ is sufficiently close to $e\in G^*$. It follows that,
when $G$ is not complete, formula \eqref{eq:action} still defines a \emph{local action} of $G$ on $\Sigma(M)$.
The proof of the Theorem shows that this action will be twisted multiplicative and hamiltonian, with momentum
map $J:\Sigma(M)\to G^*$. In other words, it is the local action given by Proposition \ref{prop:local:act}.
\end{rem}

\section{Applications}              %
\label{sec:applications}            %

In this section we will consider three applications of Theorem
\ref{thm:lift:actions}. The first application is to the
integration of a Poisson quotient $M/G$. The second application
concerns hamiltonian actions and the integration of
Lu-Weinstein quotients. The third application is to the
integration of Poisson homogeneous spaces.

\subsection{Integrability of Poisson quotients}                %
\label{subsec:quotients}                                       %

We now turn to the integrability of Poisson quotients. We start with the following remark:

\begin{prop}
\label{lift:actions:proper}
Let $G\times M\to M$ be a Poisson action of a complete Poisson-Lie group
on an integrable Poisson manifold. The lifted action $G\times\Sigma(M)\to
\Sigma(M)$ is proper (respectively, free) if and only if the
original action $G\times M\to M$ is proper (respectively, free).
\end{prop}

\begin{proof}
For the proof, we use the following simple fact: Let a Lie group $G$ act
smoothly on manifolds $P$ and $Q$ and let $\phi:P\to Q$ be a $G$-equivariant map.
If the action on $Q$ is free (respectively, proper), then the action on $P$ is free
(respectively, proper).

Now, we just need to observe that the identity section $\eps:M\to\Sigma(M)$ and the
target map $\t:\Sigma(M)\to M$ are $G$-equivariant maps.
\end{proof}

From now on we will assume that the action $G\times M\to M$ is proper and free, so that the
lifted action is also proper and free. For a Poisson action $G\times M\to M$
the space of $G$-invariant functions is a Poisson subalgebra $C^\infty(M)^G\subset C^\infty(M)$. It follows
that if the action is proper and free, this space can be identified with $C^\infty(M/G)$, so
that $M/G$ has a reduced Poisson structure $\pi_\text{red}$ such that the quotient map $M\to M/G$ is
a Poisson map.

Using the lifted action one can construct a symplectic groupoid integrating $M/G$:

\begin{thm}
\label{thm:int:M/G} Let $(G,\pi_G)$ be a complete Poisson-Lie group, $(M,\pi)$ an integrable Poisson manifold
and $G\times M\to M$ a Poisson action which is proper and free. Then the symplectic reduced space
$(J^{-1}(e)/G,\Omega_\text{red})$ is a symplectic groupoid over $M/G$ which integrates the reduced Poisson
manifold $(M/G,\pi_\text{red})$.
\end{thm}

\begin{proof}
By Proposition \ref{lift:actions:proper}, the lifted action of $G$ on $\Sigma(M)$ is proper and free. It follows
that its momentum map $J:\Sigma(M)\to G^*$ is a groupoid morphism which is a submersion. Therefore, its kernel
$J^{-1}(e)\subset\Sigma(M)$ is a Lie subgroupoid. The equivariance of the momentum map implies that $J^{-1}(e)$
is $G$-invariant. Moreover, the twisted multiplicativity property \eqref{eq:twisted} guarantees that $G$ acts on
$J^{-1}(e)$ by groupoid automorphisms. We conclude that the quotient $J^{-1}(e)/G$ is a Lie groupoid over $M/G$.

Now, by the Lu-Weinstein reduction theorem, there exists a reduced symplectic form $\Omega_\text{red}$ on
$J^{-1}(e)/G$. It follows from the multiplicativity of $\Omega$ that $\Omega_\text{red}$ is a multiplicative
2-form, so that $(J^{-1}(e)/G,\Omega_\text{red})$ is a symplectic groupoid over $M/G$ and the source map
$s:J^{-1}(e)/G\to M/G$ is a Poisson morphism. Hence, $(J^{-1}(e)/G,\Omega_\text{red})$ integrates
$(M/G,\pi_\text{red})$.
\end{proof}

Theorem \ref{thm:int:M/G} raises a natural question: does $\Sigma(M/G)$ coincide
with the symplectic reduction:
\[ \Sigma(M)//G:=J^{-1}(e)/G\, ? \]
Since we already know that this is a symplectic groupoid integrating the quotient $M/G$, the question amounts to
deciding whether the source fibers are 1-connected or not. This problem can be handled by the method used in
\cite{FerOrRa} for the case $\pi_G=0$.

First, we observe that the source fibers of $J^{-1}(e)$
need not be connected. Let $J^{-1}(e)^0$ be the
unique source connected Lie subgroupoid of $\Sigma (M)$
integrating the Lie algebroid $j^{-1}(0)\subset T^*M$:
\[
J^{-1}(e)^0=\{[a]\in \Sigma (M): j(a(t))=0,\forall t\in[0,1]\}.
\]

If $G$ is connected,  $J^{-1}(e)^0$ is $G$-invariant. The action of $G$ on $J^{-1}(e)^0$ is by automorphisms, so
we have a groupoid morphism $\Phi:J^{-1}(e)^0 \to J^{-1}(e)^0/G$ which induces the Lie algebroid morphism
$\phi:j^{-1}(0)\to j^{-1}(0)/G\cong T^*(M/G)$. On the other hand, the Lie algebroid morphism $\phi$ integrates
to a morphism of source 1-connected groupoids $\widehat{\Phi}: \G(j^{-1}(0)) \to \Sigma (M/G)$ which covers the
homomorphism $\Phi$ (here, $\G (j^{-1}(0))$ denotes the source 1-connected groupoid integrating $j^{-1}(0)$).
This yields a commutative diagram:
\begin{equation}
\label{eq:diag:com:symp}
\newdir{ (}{{}*!/-5pt/@^{(}}
\xymatrix{
K_M\ar[d]^{\widehat{\Phi}}\ar@<-2pt>@{ (->}[r]& \G(j^{-1}(0))\ar[d]^{\widehat{\Phi}}\ar[r]^{\widehat{p}}& J^{-1}(e)^0\ar[d]^{\Phi}\\
K_{M/G}\ar@<-2pt>@{ (->}[r]& \Sigma (M/G)\ar[r]^p& J^{-1}(e)^0/G}
\end{equation}
where $K_M$ and $K_{M/G}$ are group bundles over $M$ and $M/G$, respectively, with discrete fibers. Now the same
argument as in \cite{FerOrRa} shows that the group bundle $K_M$ measures how symplectization and reduction fail
to commute. More precisely, we recover \cite[Proposition 5.3]{Ste}:

\begin{thm}
\label{thm:int:reduct:alt} Let $(G,\pi_G)$ be a connected complete
Poisson-Lie group, $(M,\pi)$ an integrable Poisson manifold and
$G\times M\to M$ a Poisson action which is proper and free. Then
symplectization and reduction commute if and only if the discrete
groups
\[
  K_{m}:=\frac{\{a:I\to j^{-1}(0) \mid a \text{\rm ~is a cotangent loop such that }a\sim 0_m\}}
   {\{\text{\rm cotangent homotopies with values in } j^{-1}(0)\}}
\]
are trivial, for all $m\in M$.
\end{thm}

We refer to \cite{FerOrRa} for a detailed proof.

\subsection{Integration of hamiltonian actions}                %
\label{subsec:int:hamiltonian}                                 %

Let us turn now to the study of hamiltonian actions $G\times M\to
M$. The following remark is due to Xu \cite{Xu0}:

\begin{prop}
If the action $G\times M\to M$ is a hamiltonian action with momentum map $\mu:M\to G^*$ , then the momentum map
of the lifted action $J:\Sigma(M)\to G^*$ is exact:
\[ J(x)=\mu(\t(x)\mu(\s(x))^{-1}. \]
\end{prop}

For our next proposition we need the definition of an \emph{action groupoid} which we now briefly recall. Let
$\Psi:G\times M\to M$ be an action of a Lie group $G$ on a manifold $M$ and let $\psi:\gg\to\X(M)$ be the
corresponding infinitesimal Lie algebra action. The action groupoid associated to $\Psi$ is the groupoid whose
elements are pairs $(g,m)\in G\times M$ viewed as arrows from $m$ to $g\cdot m$. We will denote it by $G\ltimes
M$. The associated Lie algebroid, denoted by $\gg\ltimes M$, is the trivial vector bundle $\gg\times M\to M$
with anchor map $(\xi,m)\mapsto \psi (\xi )_m$ and Lie bracket characterized by
\[ [\tilde{\xi},\tilde{\eta}]_A=\widetilde{[\xi ,\eta]},\quad (\xi,\eta\in\gg)\]
where $\tilde{\xi}\in \Gamma(\gg\ltimes M)$ denotes the constant section associated
with an element $\xi\in\gg$.

One checks easily that the momentum map $\mu:M\to G^*$ defines a Lie algebroid morphism
$\psi:\gg\ltimes M\to T^*M$ by setting:
\begin{equation}
\label{eq:momentum:ham} \psi(\xi,m) = (\mu^*\xi^R)_m.
\end{equation}
Now we have:

\begin{prop}
Assume that the Lie algebroid morphism $\psi:\gg\ltimes M\to T^*M$ defined by \eqref{eq:momentum:ham} integrates
to a groupoid morphism $\Psi:G\ltimes M\to \Sigma(M)$. Then the lifted $G$-action on $\Sigma(M)$ is a twisted
inner action, i.e., it is given by:
\begin{equation}
\label{eq:action:inner}
g\,x=\Psi(g,\t(x))\cdot x\cdot \Psi(g^{J(x)},\s(x))^{-1}.
\end{equation}
\end{prop}

\begin{proof}
One checks that formula \eqref{eq:action:inner} defines a
$G$-action on $\Sigma(M)$ which is twisted multiplicative (relative
to $J:\Sigma(M)\to G^*$). The corresponding
$G$-action induced on the Lie algebroid $A=T^*M$ is just the
cotangent lifted action of $G\times M\to M$. By Proposition
\ref{prop:general}, it follows that \eqref{eq:action:inner} must
coincide with the lifted $G$-action on $\Sigma(M)$.
\end{proof}

Let us assume now that the hamiltonian action $G\times M\to M$ is proper and free. Then the quotient
$M//G:=\mu^{-1}(e)/G$ is a Poisson submanifold of $M/G$. Is this Poisson submanifold integrable? Which
symplectic groupoid integrates it? The following result gives an answer to these questions:

\begin{thm}
Let $G\times M\to M$ be a proper and free hamiltonian action with momentum map $\mu:M\to G^*$ and assume that
the Lie algebroid morphism $\psi:\gg\ltimes M\to T^*M$ given by \eqref{eq:momentum:ham} integrates to a groupoid
morphism $\Psi:G\ltimes M\to \Sigma(M)$. Then, there exists a hamiltonian action of $G\times G$ on $\Sigma(M)$
which is proper and free, and the symplectic quotient,
\[ \Sigma(M)//G\times G:=J^{-1}(e)|_{\mu^{-1}(e)}/G\times G \subset \Sigma(M)//G, \]
is a symplectic subgroupoid integrating the Poisson
submanifold $M//G\subset M/G$.
\end{thm}

\begin{proof}
We will only sketch a proof of this result. Further details will
be available in \cite{FP}.  First one defines an action of
$G\times G$ on $\Sigma(M)$ by setting:
\[ (g_1,g_2)\,x=\Psi(g_1,\t(x))\cdot x\cdot \Psi(g_2^{J(x)},\s(x))^{-1}. \]
This action is hamiltonian, with momentum map:
\[ \bar{\mu}:\Sigma(M)\to G^*\times G^*,\quad x\mapsto (\mu(\t(x)),\mu(\s(x))^{-1}).\]
Observe that the restriction of this action to the diagonal of $G\times G$ yields the lifted $G$-action.

Next one checks that the $(G\times G)$-action on $\Sigma(M)$ is proper, free, and that $(e,e)\in G^*\times G^*$
is a regular value of the momentum map. Hence, we have the symplectic quotient:
\[ \Sigma(M)//G\times G=\bar{\mu}^{-1}((e,e))/G\times G=J^{-1}(e)|_{\mu^{-1}(e)}/G\times G.\]
Note that $J^{-1}(e)|_{\mu^{-1}(e)}\subset \Sigma(M)$ is a Lie subgroupoid. One can show that its product
structure descends to the quotient $\Sigma(M)//G\times G$, so that this is a symplectic groupoid. Finally, to
complete the proof, one verifies that its source (respectively, target) map is a Poisson (respectively,
anti-Poisson) map.
\end{proof}

\begin{rem}
The quotient $M//G$ is still defined when $e\in G^*$ is a regular value of $\mu:M\to G^*$ and the action
on the level set $\mu^{-1}(e)/G$ is proper and free. In this case, one can check that the groupoid given
in the proposition stills gives an integration of $M//G$. However, now $M/G$ need not be a smooth manifold
and it may not make sense to speak of the groupoid $\Sigma(M)//G$.
\end{rem}

For a compact Poisson-Lie group $G$, a result of Ginzburg and
Weinstein \cite{GiWe} states that there is a Poisson
diffeomorphism $e:\gg^*\to G^*$. Moreover, for hamiltonian actions
of compact Poisson-Lie groups we have the following result of
Alekseev \cite{Alek}:

\begin{thm}
Let $G$ be a  1-connected, simple, compact Poisson-Lie group and let $G\times M\to M$ be a Poisson action with
momentum map $\mu:M\to G^*$. There is a Poisson structure on $M$, gauge equivalent to the original one, such
that $G$ acts by Poisson diffeomorphisms with momentum map $e^{-1}\circ\mu:M\to\gg^*$.
\end{thm}

Our description of the lifted $G$-action allows us to explain this
result at the level of the symplectic groupoid. Let us denote by
$\pi$ the original Poisson structure on $M$ and by $\tilde{\pi}$
the gauge equivalent Poisson structure. This means that $\pi$ and
$\tilde{\pi}$ have the same symplectic leaves and that there is a
closed 2-form $B\in\Omega^2(M)$ such that the symplectic
structures on a leaf $S$ differ by the pullback of $B$ to $S$ (see
\cite{BuRa}). It follows (\cite[Theorem 4.1]{BuRa}) that the
symplectic groupoids of $(M,\pi)$ and $(M,\tilde{\pi})$ have the
same groupoid structure, while the symplectic forms are related
by:
\[ \tilde{\Omega}=\Omega+\t^*B-\s^*B.\]
Our next result describes the lifted $G$-action on the symplectic
groupoid $(\Sigma(M),\tilde{\Omega})$:

\begin{thm}
For a 1-connected, simple, compact Poisson-Lie group $G$ acting on
an integrable Poisson manifold $(M,\pi)$, the $G$-action on
$(M,\tilde{\pi})$ lifts to a hamiltonian $G$-action on
$(\Sigma(M),\tilde{\Omega})$ with momentum map
$e^{-1}\circ\mu\circ\t-e^{-1}\circ\mu\circ\s:\Sigma(M)\to\gg^*$.
Moreover, this action is inner and is explicitly given by:
\[ g\,x=\Psi(g,\t(x))\cdot x\cdot \Psi(g,\s(x))^{-1}.\]
\end{thm}

The proof is more or less straightforward. We refer to \cite{FP} for more details.

\subsection{Poisson homogeneous spaces}           %
\label{subsec:Poisson:homogeneous:spaces}         %

Let $(G,\pi _G)$ be a Poisson-Lie group. We say that a Poisson manifold $(M,\pi )$ is a \emph{Poisson
homogeneous space} if there exists a Poisson action $G\times M\to M$ which is transitive (see \cite{D}). In this
section, we will give a simple description of the symplectic groupoid $\Sigma (M)$, in the case where the
Poisson structure vanishes at some point, by applying Theorem \ref{thm:lift:actions}.

Let $m_0\in M$ be the point where the Poisson structure vanishes. We identify $M$ with $G/H$, where $H=G_{m_0}$
is the isotropy group at $m_0$. Under this identification, the Poisson action of $G$ on $G/H$ is by left
translations and the projection $q:G\to G/H\cong M$ becomes a Poisson map. Then $H$ is a closed subgroup of $G$
with Lie algebra $\hh$ such that its annihilator $\hh^\perp$ is a Lie subalgebra of $\gg ^*$, that is, $H$ is a
coisotropic subgroup of $G$. In the sequel we will assume that $H$ is connected.

Our final aim is to describe the symplectic groupoid integrating the Poisson
manifold $G/H$ as some kind of quotient. This will be possible under the following
completeness assumption (compare with \cite{BCST}):

\begin{defn}
Let $G$ be a Poisson-Lie group and let $H\subset G$ be a closed,
connected, coisotropic subgroup with Lie algebra $\hh\subset\gg$.
We say that the pair $(G,H)$ is \emph{relatively complete} if
the annihilator $\hh^\perp\subset\gg^*$ integrates to a closed
subgroup $H^\perp\subset G^*$ and the right dressing action
$\gg^*\to \X(G)$ restricted to $\hh^\perp$ integrates to an action
of $H^\perp$ on $H$.
\end{defn}

As above, we let $\Sigma(G)$ be the symplectic groupoid of the Poisson manifold $G$ and we let $J:\Sigma(G)\to
G^*$ be the momentum map for the local action of $G$ on $\Sigma(G)$ obtained by lifting the action of $G$ on
itself by left translations (recall that we are not assuming that $G$ is complete, so this is only a local
action). We have:

\begin{thm}
\label{thm:Poisson:hom}
Let $G$ be a Poisson-Lie group and let $H\subset G$ be a
closed, connected, coisotropic subgroup such that the pair $(G,H)$
is relatively complete. Then $(J^{-1}(H^\perp )/H,\Omega _{\red})$
is a symplectic groupoid integrating the Poisson homogeneous space
$G/H$.
\end{thm}

\begin{proof}
The action by left translations of $G$ on itself is free, so the lifted action of $G$ on $\Sigma(G)$ is also
free and $J:\Sigma(G)\to G$ is a submersion. By assumption, $H^\perp\subset G^*$ is a closed subgroup, so it
follows that $J^{-1}(H^\perp)\subset\Sigma(G)$ is a Lie subgroupoid. Moreover, using that $J$ is a Poisson
submersion and that $H^\perp$ is coisotropic, so is $J^{-1}(H^\perp)$.

By our relative completeness assumption, the action of $H$ on $G$
by left translations can be integrated to a \emph{global action}
of $H$ on $J^{-1}(H^\perp)\subset\Sigma(G)$. This follows from the
explicit formula \eqref{eq:action}, which only uses the dressing
action of $H^\perp$ on $H$.

Now the twisted multiplicativity property \eqref{eq:twisted:2} shows that the quotient $J^{-1}(H^\perp)/H$
inherits a groupoid structure over $G/H$. This quotient is also a symplectic manifold, and the source
(respectively, target) map is a Poisson (respectively, anti-Poisson) map, so we conclude that it is a symplectic
groupoid integrating $G/H$.
\end{proof}

\begin{ex}(see also \cite[Remark 5.13]{Lu3})
Let us consider the case where $G$ is a complete Poisson-Lie group. Then the symplectic groupoid $\Sigma(G)$ is
isomorphic to the transformation groupoid $G^*\times G$, associated with the left dressing action of $G^*$ on
$G$, denoted by  $(u,g)\mapsto {}^ug$. The lift to the symplectic groupoid $\Sigma(G)$ of the action of $G$ on
itself by left translations, is given by:
\[ g\cdot (u,h)=(^gu,g^uh).\]
The momentum map $J:\Sigma(G)\to G^*$ is the projection on the factor $G^*$ and Theorem \ref{thm:Poisson:hom}
says that:
\[ J^{-1}(H^\perp )/H =(H^\perp \times G)/H,\]
is a symplectic groupoid integrating $G/H$.
\end{ex}

We do not know whether Poisson homogeneous spaces of the form $G/H$ are always integrable. Observe that $G/H$ is
integrable whenever the map $q:G\to G/H$ is a complete Poisson map. This follows from a criteria for
integrability due to Crainic and Fernandes (see \cite[Theorem 8]{CrFe2}) which states that a Poisson manifold is
integrable if and only if it admits a complete symplectic realization. For example, when $H$ is compact, the
quotient map is complete so $G/H$ is integrable, However, in general, even when $G/H$ is integrable, the pair
$(G,H)$ may fail to be relatively complete, so to construct the symplectic groupoid of $G/H$ one requires more
complicated procedures that the one given in Theorem \ref{thm:Poisson:hom}.


\end{document}